\documentclass[12pt,twoside]{amsart}
\usepackage[margin=3cm]{geometry}
\usepackage[colorlinks=false]{hyperref}
\usepackage[english]{babel}
\usepackage{graphicx,titling}
\usepackage{float}
\usepackage{amsmath,amsfonts,amssymb,amsthm}
\usepackage{lipsum}
\usepackage[T1]{fontenc}
\usepackage{fourier}
\usepackage{color}
\usepackage[latin1]{inputenc}
\usepackage{esint}
\usepackage{caption}
\usepackage{ccicons}

\makeatletter
\def\blfootnote{\xdef\@thefnmark{}\@footnotetext}
\makeatother

\newcommand\ccnote{
    \blfootnote{\ccLogo\, \ccAttribution\,\, Licensed under a Creative Commons Attribution License (CC-BY).}
}

\usepackage[export]{adjustbox}
\numberwithin{equation}{section}
\usepackage{setspace}

\renewcommand{\leq}{\leqslant}

\renewcommand{\geq}{\geqslant}
\renewcommand{\mathbb}{\varmathbb}
\usepackage{fancyhdr}
\newtheorem{theorem}{Theorem}[section]
\newtheorem{lemma}[theorem]{Lemma}

\newtheorem{proposition}[theorem]{Proposition}
\newtheorem{remark}[theorem]{Remark}
\usepackage{enumitem}
\usepackage[abbrev,lite,nobysame]{amsrefs}

\usepackage{mathtools}
\mathtoolsset{showonlyrefs=true}

\usepackage{dsfont}
\providecommand{\ip}[1]{\langle#1\rangle}
\providecommand{\abs}[1]{\left\lvert#1\right\rvert}
\providecommand{\norm}[1]{\left\|#1\right\|}

\def\eps{\varepsilon}
\def\e{{\rm e}} 
\def\dd{{\rm d}}
\def\ddt{{\frac{\dd}{\dd t}}}
\def\R {\mathbb{R}}
\def\Z {\mathbb{Z}}
\def\N {\mathbb{N}}

\def\T {{\mathbb T}}

\def\de{{\partial}}

\def\cL {\mathcal{L}}

\address{Maria Colombo, Institute of Mathematics, EPFL, Station 8, 1015 Lausanne, Switzerland}
\email{maria.colombo@epfl.ch}
\medskip
\address{Michele Coti Zelati, Department of Mathematics, Imperial College London, London, SW7 2AZ, UK}
\email{m.coti-zelati@imperial.ac.uk}
\address{Klaus Widmayer, Institute of Mathematics, EPFL, Station 8, 1015 Lausanne, Switzerland}
\email{klaus.widmayer@epfl.ch}

\begin{document}

\thispagestyle{empty}

\begin{minipage}{0.28\textwidth}
\begin{figure}[H]
%\centering
\includegraphics[width=2.5cm,height=2.5cm,left]{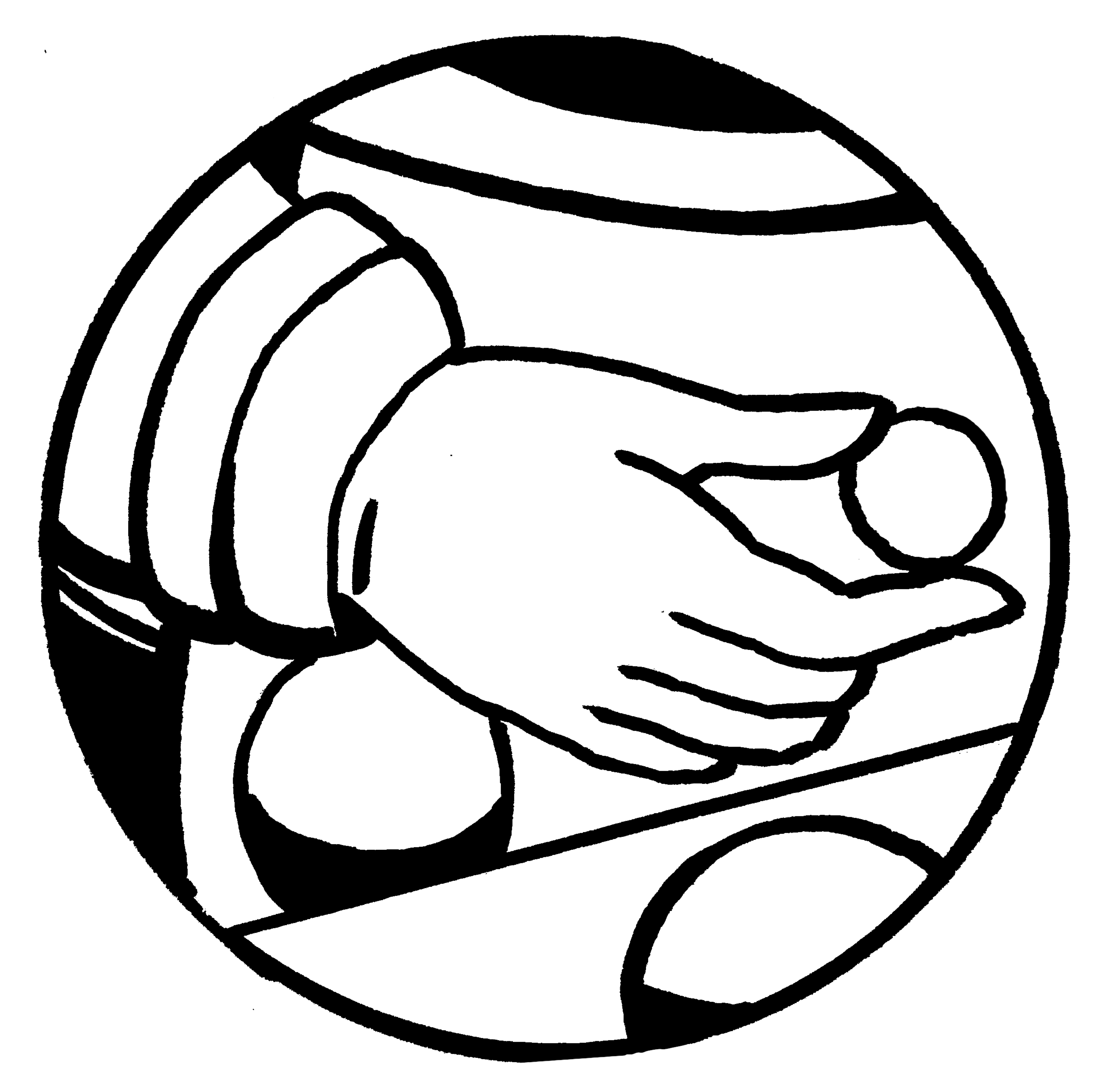}
\end{figure}
\end{minipage}
\begin{minipage}{0.7\textwidth} 
\begin{flushright}
%% The following metadata, in particular
%% the Paper No. and the DOI will be inserted by the journal
Ars Inveniendi Analytica (2021), Paper No. 2, 21 pp.
\\
DOI 10.15781/83fc-j334
\end{flushright}
\end{minipage}

\ccnote

\vspace{1cm}

%%      -------------------------------------------------------------------------------
%%      -------------------------- TITLE ----------------------------
%%      -------------------------------------------------------------------------------
%% Authors, please put here the full title of the article

\begin{center}
\begin{huge}
\textit{Mixing and diffusion for rough shear flows}

%\textit{some titles take two lines}

\end{huge}
\end{center}

\vspace{1cm}

%%      -------------------------------------------------------------------------------
%%      -------------------------- AUTHORS AND AFFILIATIONS ----------------------------
%%      -------------------------------------------------------------------------------
%% Authors, please put here your full names and affiliations

\begin{minipage}[t]{.28\textwidth}
\begin{center}
{\large{\bf{Maria Colombo}}} \\
\vskip0.15cm
\footnotesize{EPFL Lausanne}
\end{center}
\end{minipage}
\hfill
\noindent
\begin{minipage}[t]{.28\textwidth}
\begin{center}
{\large{\bf{Michele Coti Zelati}}} \\
\vskip0.15cm
\footnotesize{Imperial College London}
\end{center}
\end{minipage}
\hfill
\noindent
\begin{minipage}[t]{.28\textwidth}
\begin{center}
{\large{\bf{Klaus Widmayer}}} \\
\vskip0.15cm
\footnotesize{EPFL Lausanne} 
\end{center}
\end{minipage}

\vspace{1cm}

%%% Please replace "James Mustard" below 
%%% with the name of the managing editor for your submission.
%%% If you are unsure about their identity
%%% please ask an editor-in-chief about.

\begin{center}
\noindent \em{Communicated by Philip Isett}
\end{center}
\vspace{1cm}

%%      -------------------------------------------------------------------------------
%%      -------------------------- BEGIN ABSTRACT ----------------------------
%%      -------------------------------------------------------------------------------
%% Authors, please put here the ABSTRACT and KEYBOARDS

\noindent \textbf{Abstract.} \textit{This article addresses mixing and diffusion properties of passive scalars advected by rough ($C^\alpha$) shear flows. We show that in general, one cannot expect a rough shear flow to increase the rate of inviscid mixing to more than that of a smooth shear. On the other hand, diffusion may be enhanced at a much faster rate. This shows that in the setting of low regularity, the interplay between inviscid mixing properties and enhanced dissipation is more intricate, and in fact contradicts some of the natural heuristics that are valid in the smooth setting.
}
\vskip0.3cm

\vspace{0.5cm}

%%      -------------------------------------------------------------------------------
%%      -------------------------- BEGIN ARTICLE ----------------------------
%%      -------------------------------------------------------------------------------
%% Authors, copy the body of your paper here

\section{Introduction}
This paper is concerned with the long-time behavior of the linear equation
\begin{align}\label{eq:passcal}
\begin{cases}
\de_t f +u\de_x f=\nu\Delta f \\
f(0)=f^{in} \qquad \qquad \int_\T f^{in}(x,y)\dd x=0,
\end{cases}
\end{align}
posed on the two-dimensional square torus $\T^2$. In \eqref{eq:passcal}, $f=f(t,x,y):[0,\infty)\times \T^2\to\R$ 
represents a passive scalar that is advected by 
a shear flow $u=u(y):\T\to\R$, and $\nu\in[0,1]$ is the diffusion coefficient that, when strictly positive, induces a dissipation
mechanism (due to molecular friction, for instance). Here the condition that $\int_\T f^{in}(x,y)\dd x=0$ naturally ensures that $f$ witnesses the effect of the transport operator.

We are interested in understanding the \emph{inviscid mixing} and \emph{enhanced diffusion} properties of $f$ in the case when $u$ is not smooth,
and in particular in the case of H\"older continuous flows. When \eqref{eq:passcal} is considered in its inviscid form with $\nu=0$, 
the term inviscid mixing refers to a transfer of energy
from large to small spatial scales for the scalar $f$. If a small amount of diffusion is introduced by taking a strictly positive $\nu\ll1$, 
the inviscid mixing mechanism is still present for relatively short times 
and the dissipation is therefore necessarily enhanced.

Our main result addresses the connection of the \emph{regularity} of the transporting vector field $u$ with the rate of such mixing effects.
We construct a large family of $C^\alpha$ shear flows which inviscidly mix at the same, fixed rate as a smooth shear flow without critical points, but for which the enhanced diffusion decay becomes arbitrarily fast
as $\alpha\to 0$.

\begin{theorem}\label{thm:main}
 There exists a dense set $\mathcal{A}\subset (0,1)$ such that for each $\alpha\in\mathcal{A}$ there exists a function $u\in C^\alpha(\T)$ which is sharply $\alpha$-H\"older, with the following properties: {all} solutions {$f(t)$} to the passive scalar problem \eqref{eq:passcal} satisfy
 \begin{enumerate}
  \item $\nu=0$: inviscid mixing, in the sense that there exists $C=C(u)>0$ such that 
  \begin{enumerate}
   \item for all $t>0$ there holds that 
   \begin{equation}\label{eq:mixing}
    \norm{f(t)}_{L^2_xH^{-1}_y}\leq C t^{-1}\norm{f^{in}}_{H^{-1}_xH^1_y},
   \end{equation}
   \item there exists a monotone sequence of times $(t_m)_{m\in\N}$, with $t_m\to\infty$, $m\to\infty$, along which the decay is faster, i.e.\ there holds
   \begin{equation}\label{eq:mixing_fast}
    \norm{f(t_m)}_{L^2_xH^{-1}_y}\leq Ct_m^{-\frac{1}{\alpha}}\norm{f^{in}}_{L^2_xH^1_y},
   \end{equation}
   \item {Moreover}, the decay \eqref{eq:mixing} is sharp {in the following sense}: {There exist initial data $f^{in}_\star$ with $\int_\T f^{in}_\star\dd x=0$} and a monotone sequence of times $(t'_m)_{m\in\N}$, with $t'_m\to\infty$ as $m\to \infty$
   such that the corresponding solution $f_\star(t)$ of {\eqref{eq:passcal}} satisfies
    \begin{equation}\label{eq:mixing_sharp}
    \norm{f_\star(t'_m)}_{L^2_xH^{-1}_y}\geq C(t'_m)^{-1}\norm{f^{in}_\star}_{L^2_xH^1_y}.
    \end{equation}
  \end{enumerate}

  \item $\nu>0$: dissipation is enhanced, in the sense that there exists $\eps>0$ such that
    \begin{align}\label{eq:enhance}
   \norm{f(t)}_{L^2}\leq C \e^{-\eps \nu^{\frac{\alpha}{\alpha+2}}t}\norm{f^{in}}_{L^2}.
  \end{align}
 \end{enumerate}
\end{theorem}

\begin{remark}
{Due to the shear structure of the velocity field, the equation \eqref{eq:passcal} decouples in Fourier space in the $x$-frequency. In particular, estimate \eqref{eq:enhance}  can be phrased in a frequency-localized fashion  as $ \norm{f_k(t)}_{H^{-1}_y}\leq C (|k|t)^{-1}\norm{f_k^{in}}_{H^1_y}$, with $f_k$, for $k\in \Z\setminus\{0\}$, being the $k$-th Fourier mode in $x$ of $f$.  
This gain of one derivative in $x$ leads to the more general bound
   \begin{equation}
    \norm{f(t)}_{H^\sigma_xH^{-1}_y}\leq C t^{-1}\norm{f^{in}}_{H^{\sigma-1}_xH^1_y},
   \end{equation}
  for any $\sigma \in \R$.
} 
\end{remark}

For the statements with full details we refer the reader to Propositions \ref{prop:inv_decay} and \ref{prop:enhdissip} below. The negative Sobolev norm $H^{-1}$ in \eqref{eq:mixing}, introduced in the context of inviscid mixing in \cite{LTD11} and 
now widely used \cite{CLS19,BM19mix,ACM19,EZ19,Seis13,IKX14,ACM19-2,LuLietc12}, 
precisely describes how energy is transferred to small scales (or high frequencies), and is sometimes referred to as mixing norm. Theorem~\ref{thm:main} can be stated also in terms of the  geometric mixing rate of these $C^\alpha$ shear flows;  {this is carried out in \cite{CSprep} and  links our work with Bressan's conjecture on mixing by non-smooth vector fields \cite{Bressan03}.}
When $\nu>0$,
the effect of enhanced dissipation \cite{CKRZ08}, can be quantified in terms of decay estimates on the $L^2$ norm of $f$, as in \eqref{eq:enhance}. 
Limiting ourselves to the case of passive scalars, we mention the recent works \cite{FI19,BN20,CZDE18,Seis20,BCZ17,CZ20,CZD19,WEI18,BW13,ABN2021}.

The surprising feature of Theorem \ref{thm:main} is that there is a discrepancy between the inviscid mixing and enhanced dissipation rates: While the roughness of $u\in C^\alpha$ is witnessed directly in the viscous problem with the fast, $\alpha$ dependent and sharp (see \cite{CZDri19}) rate proportional to $\nu^{\frac{\alpha}{\alpha+2}}$, the inviscid mixing generally only happens at the fixed rate $1/t$. In a smooth setting, this latter fact would suggest an enhanced diffusion decay rate proportional to $\nu^{1/3}$ (see also Section \ref{sub:diff}), which is exactly what happens in the Couette flow \cite{BGM19}, where $u(y)=y$, on the spatial domain $\T\times \R$. Nevertheless, we note that the faster rate $t^{-\frac{1}{\alpha}}$, which would more naturally correspond to fast diffusion, can still be seen in the inviscid problem along a sequence of times.

\subsection{Context of our result}\label{sub:diff}
To the best of our knowledge, there have been no prior works investigating the connection between low regularity and mixing rates. However, there is a body of work \cite{BCZ17,CZ20,BW13,CZDri19} regarding the connection of inviscid and viscous mixing rates in the setting of smooth shear flows. In this context, the following informal heuristic seems to correctly predict the relationship between inviscid mixing and enhanced dissipation (where it is known): Let us assume that for $\nu=0$, solutions to \eqref{eq:passcal}, {possibly localized in a single $x$-frequency shell}, obey the {(slightly simplified)} inviscid mixing estimate
\begin{equation}\label{eq:Heurmixing}
  \| f(t)\|_{H^{-1}}\leq \rho(t)  \| f^{in}\|_{H^{1}}, \qquad \forall t\geq 0,
\end{equation}
for a monotonically decreasing function $\rho:[0,\infty)\to[0,\infty)$ vanishing at infinity. 

For small positive $\nu$ and for sufficiently short time, one might expect this mechanism to still be the leading order effect, so that the dissipation $\nu \Delta f$ translates {in the energy estimates of \eqref{eq:passcal} to a damping term, i.e.\
\begin{equation}
 \frac{1}{2}\ddt\norm{f(t)}_{L^2}^2=-\nu \norm{\nabla f}_{L^2}^2\lesssim -\nu [\rho(t)]^{-2}\norm{f(t)}^2_{L^2}.
\end{equation}
This toy model thus} gives the decay estimate 
\begin{align}
\| f(t)\|_{L^2}\leq \e^{- C \nu \int_0^t [\rho(s)]^{-2}\dd s} \| f^{in}\|_{L^2},
\end{align}
so that one is tempted to at least guess the enhanced dissipation \emph{time scale} $\tau_\nu$ from the computation of $\nu \int_0^t [\rho(s)]^{-2}\dd s$.
For instance, for any $\beta>0$ we have
\begin{align}\label{eq:rateguess}
\rho(t)\sim t^{-\frac1\beta} \quad \Rightarrow \quad \nu \int_0^t [\rho(s)]^{-2}\dd s\sim \left(\nu^\frac{\beta}{\beta+2} t\right)^\frac{\beta+2}{\beta}\quad \Rightarrow\quad  \tau_\nu\sim \nu^{-\frac{\beta}{\beta+2}}.
\end{align}
While the super-exponential decay that is deduced by this heuristic argument should be ignored (at least whenever the corresponding linear operator in \eqref{eq:passcal} has non-empty spectrum), 
an inviscid mixing rate proportional to $t^{-\frac1\beta} $ would then correspond to an enhanced dissipation time-scale proportional to 
$\nu^{-\frac{\beta}{\beta+2}}$ (quantified in a way similar to \eqref{eq:enhance}). 

To the best of our knowledge, there is no rigorous result that fully justifies this computation, 
even in the case of smooth velocity fields. The only rigorous derivation of enhanced dissipation rates from mixing ones has been carried out in \cite{CZDE18} for a very general class of evolution problems of which \eqref{eq:passcal} is a particular case (see also \cite{FI19} for related results). However, the time scale deduced there is worse than the one predicted in \eqref{eq:rateguess}.

Nevertheless, it is remarkable that the heuristic \eqref{eq:rateguess} is in agreement with the known results for smooth velocity fields, namely for regular shear flows \cite{BCZ17,CZ20} and radial flows \cite{CZD19}, and velocity fields arising from the stochastic Navier-Stokes equations \cite{BBPS19a,BBPS19b} . In addition, for these the rates are known to be sharp \cite{CZDri19,BBPS19a,BBPS19b}. However, in the deterministic setting, the techniques used to analyze the inviscid and viscous problems are very different (typically oscillatory integral \cite{CZDE18,WZZ19} versus spectral theoretic \cite{SMM19,WZZ20} or hypocoercivity methods \cite{BCZ17,WZ19, CZ20,CZEW20}) and do not make use of the above ``connection''.

In view of this, our result Theorem \ref{thm:main} shows that while there may be hope to better understand the mechanisms behind the above heuristics for smooth velocity fields, in the realm of less regular flows this is not the case. The family of $C^\alpha$ flows we construct have uniform inviscid mixing rate $t^{-1}$ (see \eqref{eq:mixing} and \eqref{eq:mixing_sharp}), but nonetheless enhance dissipation at a faster rate proportional to $\nu^{\frac{\alpha}{\alpha+2}}$ (see \eqref{eq:enhance}). However, the inviscid decay may still be the faster $t^{-\frac{1}{\alpha}}$ along a sequence of times (see \eqref{eq:mixing_fast}).

We conclude that in the setting of rough flows, the interplay of inviscid mixing and enhanced dissipation is more intricate, and to a large extent remains to be understood. 

\subsection*{Perspectives}
Naturally, one may wonder to what extent (a version of) the heuristics \eqref{eq:rateguess} can be rigorously established in any setting. While we show that this is not possible for rougher flows, a general picture that connects regularity and mixing properties is still elusive. Even in the category of $C^\alpha$ flows it is not clear whether our result is generic in any sense, or whether other behaviors are to be expected.

\subsection{About the proof}
The basic mechanism behind the inviscid estimate \eqref{eq:mixing} is that oscillations lead to cancellation, which is typically exploited by integrations by parts in stationary phase type lemmas. Such results, namely estimates of the form 
\begin{align}\label{eq:statphase}
\int_\T \e^{itv(y)}\varphi(y)\dd y\leq Ct^{-\frac {1}{\beta}}\norm{\varphi'}_{L^{1}}, \qquad \forall t\in \R,
\end{align}
are classical in a context different from the one treated here: namely, when $\beta=n_0$ is a positive integer and the function $v$ belongs to $C^{n_0}$ and has no vanishing points of order $n_0$ (i.e.\ no points where all first $n_0$ derivatives vanish).

In our setting, $u \in C^\alpha$ and satisfies 
a lower H\"older estimate of order $\alpha$ (see Lemma \ref{lem:lowerHolder}): although this could be read as saying that ``the function $u$ is vanishing to a fractional order $\alpha$'', corresponding to the estimate \eqref{eq:statphase} with $\beta=\alpha$, the scenario is in fact more varied. 
By a careful choice of the function $u$ and of a sequence of times, we discover on one side, that in general no better rate than $\beta=1 $ in \eqref{eq:statphase} can be expected; on the other side, we show that we may get the faster rate $\beta= \alpha $ in \eqref{eq:statphase} for another particular sequence of times.

To prove this statement, we need to perform a precise construction of the family of shear flows $u\in C^\alpha$ in Section \ref{sec:constr}. It uses a natural, explicit limiting procedure,
in which $u$ is approximated by piecewise linear functions $u_m$ which iteratively increase the ``microstructure'' in a self-similar fashion, {akin to the classical construction of Weierstrass \cite{BPbook}.} From this, the sharp $C^\alpha$ property and some other, useful attributes of $u$ are deduced (see also Lemmas \ref{lem:lowerHolder} and \ref{lem:Ssplit}).
To obtain a stationary phase estimate as \eqref{eq:statphase}, while it is easy to believe that $u$ ``oscillates wildly'', in terms of integrations by parts one only has access to this feature at the level of the approximating functions $u_m$. {Here we use a cancellation property specific to our construction -- a property which is not shared by a general Weierstrass-type function.} Due to the fractal structure of these, a delicate analysis is required to deduce decay properties \eqref{eq:mixing} (see Section \ref{sec:inviscidmixing} and in particular Proposition \ref{prop:inv_decay}). Similar ideas then yield the decay bounds \eqref{eq:mixing_fast} and \eqref{eq:mixing_sharp} along sequences of time (see also Lemma \ref{lem:inv_decay_new}).

We finally remark that a faster speed than $\beta=\alpha$ is not expected in \eqref{eq:statphase}, since we prove in Proposition~\ref{prop:opt} that the decay $t^{-\frac 1 \alpha}$ obtained in \eqref{eq:mixing_fast} is essentially optimal in the context of Theorem~\ref{thm:main}.

Our derivation of the enhanced diffusion estimate \eqref{eq:enhance} has instead a spectral theoretic flavor, and leans on 
a criterion devised in \cite{WEI18} to estimate enhanced dissipation rates of shear flows given by Weierstrass functions.
Our proof in Section \ref{sec:enhanced} uses crucially the construction of the shear flow $u$ as a limit of piecewise linear functions.

\subsection{Notation}
For $r,s\geq -1$ we denote by $H^r_xH^s_y$ the standard Sobolev spaces on $\T^2$ with norms $\norm{\varphi}_{H^r_xH^s_y}=\norm{\ip{k}^r\ip{n}^s\widehat{\varphi}(k,n)}_{L^2}$, where $\ip{k}=(1+k^2)^{1/2}$ for $k\in\Z$. Whenever $r=s$ we simply write $H^s$. Furthermore, we note that by duality we may also compute
\begin{equation}
 \norm{\varphi}_{L^2_xH^{-1}_y}^2=\sum_{k\in\Z}\sup_{\psi\in H^1_y,\norm{\psi}_{H^1}\leq 1}\abs{\ip{(\mathcal{F}_{x\to k}\varphi)(k,y),\psi}_{L_y^2}}^2.
\end{equation}

\section{Construction of the shear flows}\label{sec:constr}
Let $u_0$ be a periodic tent function on $[-\pi,\pi]$
\begin{equation}
 u_0(y):=\begin{cases}
         1+\frac{1}{\pi}y, &y\in[-\pi,0],\\
         1-\frac{1}{\pi}y, &y\in[0,\pi].
        \end{cases}
\end{equation}
Next we inductively define periodic, continuous, piecewise linear functions $u_m$, $m\in\N$, as subsequent oscillations around $u_0$ as follows (see also Figure \ref{fig:um}):

\begin{figure}[h!]
  \includegraphics[scale=0.7]{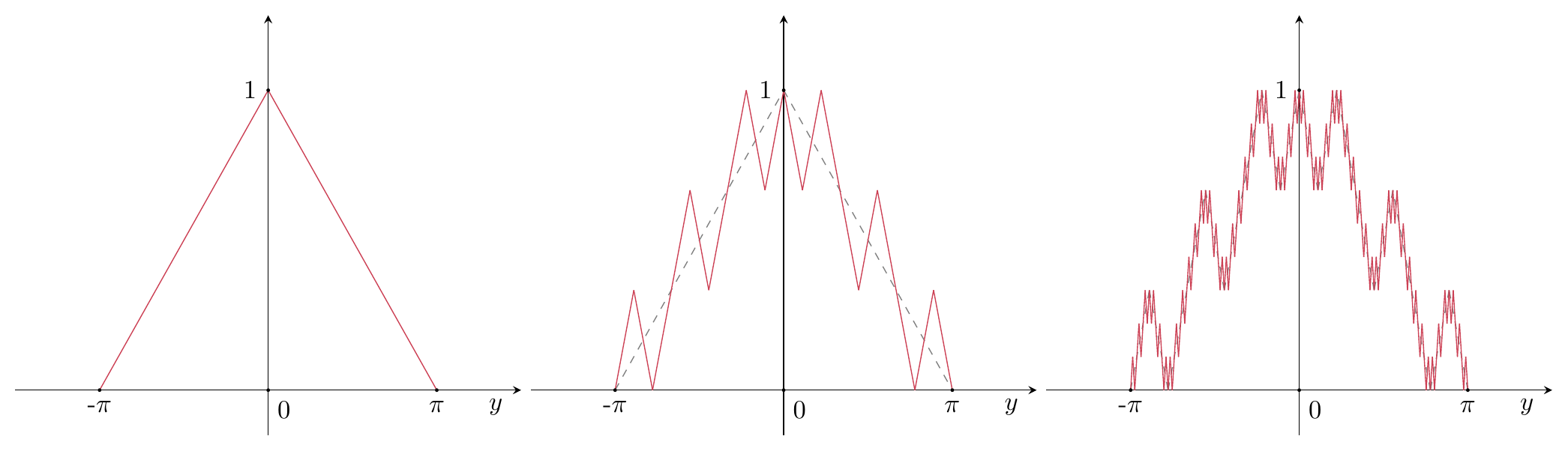}
  \caption{$u_0$, $u_1$ and $u_2$ for $p=q=3$.}
  \label{fig:um}
\end{figure}

Let \emph{odd} numbers $p,q\geq 3$ be given.\footnote{The choice of $p\geq 3$ is a matter of technical convenience -- we could also allow $p=2$ at the cost of some minor, but technical, modifications to some of the later proofs.} For $m\in\N$ we subdivide $[-\pi,\pi]$ into $N_m=2 p^mq^m$ intervals of equal length
\begin{equation}\label{eq:lm}
 \ell_m=\frac{2\pi}{N_m}=\frac{\pi}{p^mq^m},
\end{equation}
thus obtaining a grid
\begin{equation}
 y_m^j=-\pi +j \ell_m, \qquad j=0,\ldots, N_m.
\end{equation}
Now we inductively define the continuous, piecewise linear functions $u_m$ as follows. For some $m\in\N$, let $u_m$ be given satisfying that 
\begin{enumerate}
 \item $u_m$ is continuous and piecewise linear,
 \item $u_m$ is linear on each interval $I_m^j:=[y_m^j,y_m^{j+1}]$, $j=0,\ldots,N_m-1$.
\end{enumerate}
Now we subdivide each interval $I_m^j$ into $p$ subintervals $I_m^{j,k}:=[y_m^{j,k},y_m^{j,k+1}]$, where $y_m^{j,k}:=y_m^j+kh_{m+1}$, with $0\leq k\leq p-1$, each of length $h_{m+1}:=\ell_m p^{-1}=q\ell_{m+1}$. On each such subinterval $I_m^{j,k}$ we then define $u_{m+1}$ as piecewise linear, continuous and oscillating $q$ times between the values $u_m(y_m^{j,k})$ and $u_m(y_m^{j,k+1})$, i.e.\ we define
\begin{equation}\label{eq:defum}
 u_{m+1}(y):=\begin{cases}
 u_m(y_m^{j,k}), &y\in\{y_m^{j,k}+2a\ell_{m+1};\,a=0,1,\ldots,\frac{q-1}{2}\},\\
 u_m(y_m^{j,k+1}), &y\in\{y_m^{j,k}+(2a+1)\ell_{m+1};\,a=0,1,\ldots,\frac{q-1}{2}\},\\
 \textnormal{piecewise linear} &\textnormal{in between}.
 \end{cases}
\end{equation}
The functions $u_m$ so constructed have slope of size
\begin{align}\label{eq:umder}
 s_m:=|u'_m(y)|=\frac{1}{\pi} q^m, \qquad \forall y\in\T\setminus\{y_m^j;\,j=0,\ldots,N_m\}.
\end{align}
The following lemma proves that $u\in C^\beta$ if and only if $\beta \leq \alpha$. More precisely the sharpness of the H\"older exponent follows from the so-called lower-H\"older property (see \eqref{eq:sharpCalpha}).
\begin{lemma}\label{lem:lowerHolder}
The family of functions $\{u_m\}_{m\in\N}$ converges uniformly to a function $u\in C^{\alpha}(\T)$, where
\begin{align}\label{eq:alpha}
\alpha=\frac{\ln p}{\ln p+\ln q}.
\end{align}
Moreover, there exists a constant $C>0$ such that for any interval $I\subset \T$ of length 
$0<\abs{I}\leq 1$ 
\begin{equation}\label{eq:sharpCalpha}
 \sup_I u - \inf _I u \geq C\abs{I}^\alpha.
\end{equation}

\end{lemma}

\begin{remark}
 We note that the set of $\alpha\in(0,1)$ as in \eqref{eq:alpha} is dense in $(0,1)$ because it contains $\mathbb Q$: Given $a<b\in\N$, we choose   $p=k^a$, $q=k^{b-a}$ for some $k\in\N$ odd to obtain that $\frac{\ln p}{\ln p+\ln q}=\frac{a}{b}$.
\end{remark}

\begin{proof} By construction, for each $m\in \N$ and $y\in \T$ we have that
\begin{align}
|u_m(y)-u_{m+1}(y)|\leq \frac{1}{p^{m+1}}.
\end{align}
Hence,
for any $k\in \N$,
\begin{align}
\|u_m-u_{m+k}\|_{L^\infty}\leq \sum_{\ell=0}^{k-1} \|u_{m+\ell}-u_{m+\ell+1}\|_{L^\infty}\leq \sum_{\ell=0}^{k-1} \frac{1}{p^{m+\ell+1}}\leq
\frac{1 }{(p-1)p^{m}}.
\end{align}
Consequently, $\{u_m\}_{m\in\N}$ is uniformly Cauchy and converges to a continuous function $u$ and
\begin{align}\label{eq:convest}
\|u_m-u\|_{L^\infty}\leq\frac{1 }{(p-1)p^{m}}.
\end{align}
To show that $u$ is H\"older-continuous,
let us consider two points $y,\bar{y}\in \T$, and take $m\geq 0$ such that
\begin{align}
2\ell_{m+1}\leq |y-\bar{y}|\leq 2\ell_m.
\end{align}
Thanks to \eqref{eq:convest} and to \eqref{eq:umder}, we have that
\begin{align}\label{eq:holder2}
|u(y)-u(\bar{y})|\leq |u_m(y)-u_{m}(\bar{y})|+\frac{2 }{(p-1)p^{m}}\leq 2s_m\ell_m+\frac{2}{(p-1)p^{m}}\leq \frac{2 p}{(p-1)p^{m}}.
\end{align}
Thanks to our choice of $\alpha$ in \eqref{eq:alpha}, we have that $(pq)^\alpha=p$, namely $\frac{1}{p^{m}}
=\frac{p}{2^\alpha\pi^\alpha}\left(2\ell_{m+1}\right)^\alpha
$ and from \eqref{eq:holder2} we deduce our claimed estimate
\begin{align}
|u(y)-u(\bar{y})|\leq \frac{2^{1-\alpha} p^2}{\pi^\alpha(p-1)}\left(2\ell_{m+1}\right)^\alpha \leq   \frac{2^{1-\alpha} p^2}{\pi^\alpha(p-1)}|y-\bar{y}|^\alpha.
\end{align}
To see the sharp H\"older property \eqref{eq:sharpCalpha}, let $I\subset\T$ be given. 
 Now let $m\in\N$ such that $\ell_{m-1}\leq\abs{I}<\ell_{m-2}$. Then by construction there exists $j\in\{0,...,N_m\}$ such that $[y_m^j,y_m^{j+1}]\subset I$. Now we note that $u(y_m^j)=u_m(y_m^j)$ and $u(y_m^{j+1})=u_m(y_m^{j+1})$, so by \eqref{eq:umder} it follows that
\begin{equation}
 \abs{u(y_m^j)-u(y_m^{j+1})}=s_m\ell_m=\frac{1}{p^m} =\frac{1}{p^3} \frac 1 {p^{m-3}} =\frac{1}{\pi^\alpha p^2} \ell_{m-2}^\alpha \geq \frac{1}{\pi^\alpha p^2}  |I|^\alpha
\end{equation}
where we used \eqref{eq:umder}, \eqref{eq:lm} and again that $p=(pq)^\alpha$.
\end{proof}

\section{Inviscid mixing properties}\label{sec:inviscidmixing}
Our main goal in this section is to establish the following result:
\begin{proposition}\label{prop:inv_decay}
 Let $p,q, \alpha, u$ as in Section~\ref{sec:constr}. Then there exists $C:= C(p,q)>0$ such that the following hold:
 \begin{enumerate}
  \item\label{it:1} 
  If $f$ solves
 \begin{equation}\label{eq:pass_scal}
  \de_t f+u\de_x f=0,\qquad f(0)=f^{in},\qquad \int_\T f^{in}(x,y)\dd x=0,
 \end{equation}
 then
 \begin{equation}\label{eq:inv_decay_gen}
  \norm{f(t)}_{L^2_xH^{-1}_y}\leq Ct^{-1}\norm{f^{in}}_{H^{-1}_xH^1_y}.
 \end{equation}
 \item\label{it:2} This decay is sharp, in the sense that {there exist initial data $f^{in}_\star$ with $\int f^{in}_\star\dd x=0$} such that the corresponding solution $f_\star(t)$ of \eqref{eq:pass_scal} satisfies {for $t'_m=\pi p^m$, $m\in\N$, that}
 \begin{equation}\label{eq:inv_decay_lobd}
  \norm{f_\star(t'_m)}_{L^2_xH^{-1}_y}\geq C(t'_m)^{-1}\norm{f^{in}_\star}_{L^2_xH^1_y}.
 \end{equation}

 \item\label{it:3} However, the decay along a sequence of times may be significantly faster than this: We have that for $t_m=2\pi p^m$, $m\in\N$, any solution of \eqref{eq:pass_scal} satisfies
 \begin{equation}\label{eq:inv_decay_fast_m}
  \norm{f(t_m)}_{L^2_xH^{-1}_y}\leq Ct_m^{-\frac{1}{\alpha}}\norm{f^{in}}_{L^2_xH^1_y}.
 \end{equation}
 \end{enumerate}
\end{proposition}

Parts \eqref{it:1} and \eqref{it:2} show that the shears induced by the functions $u$ of Section \ref{sec:constr} in general only mix at the same rate as smooth shears without critical points {(in a periodic channel or strip), and only slightly faster than sinusoidal profiles, whose mixing rate is proportional to $t^{-1/2}$, see \cite{BCZ17}.} Part \eqref{it:3} shows that along a sequence of times this may be much faster. We highlight here that the particular sequence of times does not depend on the choice of initial data, but is constructed based on the structure of $u$. 

The idea of the proof of these results is well known: oscillations lead to cancellation, which can typically be seen via an integration by parts. While $u$ is certainly constructed to have plenty of oscillations, here we cannot directly exploit them via integration by parts on $u$, since $u\in C^\alpha$ is not regular enough. Naturally, though, the approximating piecewise linear functions $u_m$ have increasingly steep slopes (as $m\to\infty$). However, $u_m'$ is only linear on increasingly many (namely, $N_m$) increasingly small (of size $N_m^{-1}$) intervals and changes sign frequently, so a delicate analysis is needed to show that this actually leads to \emph{uniform in $m$} bounds for the rate of mixing of $u_m$ (see also Lemma \ref{lem:inv_decay}). This gives \eqref{eq:inv_decay_gen}, and similar ideas can be used to establish \eqref{eq:inv_decay_lobd} and \eqref{eq:inv_decay_fast_m}.
After a discussion of the relevant properties of $u_m$ in Section \ref{ssec:prelim_inv}, we then give the details of these proofs in Section \ref{ssec:pf_inv} below.

\subsection*{On the sharpness of the decay} It is worth noting that by extending heuristics of the smooth case (see also our discussion in the introduction), the faster rate $t^{-\frac{1}{\alpha}}$ may be regarded as a natural scale for inviscid mixing of a $C^\alpha$ shear. Moreover, in a certain sense it is also the fastest rate possible, as a simple computation shows: 
\begin{proposition}\label{prop:opt}
 Let $p,q, \alpha, u$ as in Section~\ref{sec:constr} and let $f^{in}\in W^{1,2}(\T^2)$. If $f$ solves \eqref{eq:pass_scal},
  then for every $\alpha' <\alpha$ there exists a constant $\bar{C}>0$ such that
 \begin{equation}\label{ts:opt}
  \norm{f(t)}_{H^{-1}(\T^2)}\geq \bar{C} t^{-\frac{1}{\alpha'}} \qquad \mbox{for every } t\geq 1.
 \end{equation}
\end{proposition}
We note, however, that the constant $\bar{C}$ in \eqref{ts:opt} may depend on $f^{in}$ -- see \eqref{eq:alpha'lowbd} for the details. {Concerning the case of $\alpha'=\alpha$, in the smooth setting such a lower bound { may be false in general} due to orthogonality and instead requires some mild additional decay \cite{Zil19}.}
\begin{proof}
Let $X_t(x,y)= (x-tu(y), y):[0,\infty) \times \T^2 \to \T^2$ be the flow of the stationary vector field $(-u(y),0)$. Observe that $\| X_t\|_{C^\alpha} \leq C t$ for every $t \geq 1$, where $C$ depends only on $\|u\|_{C^\alpha}$. The solution $f(t)$ can be represented as 
$f(t,x,y) = f^{in}(X_t(x,y))$. Hence by the Hajlasz inequality \cite{Stein70}
$$ |f(t,z)-f(t,z')| \leq \big( |M{(}Df^{in}{)}|(X_t(z))+ |M(Df^{in})|(X_t(z'))\big) | X_t(z)-X_t(z')| \leq C (1+t) |z-z'|^\alpha,$$
{where $M$ is the maximal function.}
Hence, for any $\alpha'<\alpha$, $t\geq 1$ by a direct computation we can estimate the Gagliardo seminorm of $f(t)$ \begin{equation}
\begin{split}
\|f(t)\|_{W^{\alpha', 2}}^2 &= \int_{\T^2}\int_{\T^2} \frac{|f(t,z)-f(t,z')|^2}{|z-z'|^{2+2\alpha'}} \dd z\dd z' \\
&\leq\int_{\T^2}\int_{\T^2} \frac{| |M(Df^{in})|^2(X_t(z))+ |M(Df^{in})|^2(X_t(z'))|^2}{|z-z'|^{2+2(\alpha'-\alpha)}} \dd z\dd z'\leq C \|f^{in}\|^2_{H^{1}} {(1+t^2)},
\end{split}
\end{equation}
where in the last inequality we used that the kernel $|z-z'|^{-2-2(\alpha'-\alpha)}$ is integrable, that $X_t$ is measure preserving and that the maximal function maps $L^2$ to $L^2$ with a linear estimate. The constant $C$ depends on $\alpha'$ and on $\|u\|_{C^\alpha}$.
By interpolation
\begin{equation}\label{eq:alpha'lowbd}
\|f^{in}\|_{L^2}^2= \|f(t)\|_{L^2}^2 \leq  \|f(t)\|_{H^{-1}}^{\frac{2\alpha'}{1+\alpha'}}   \|f(t)\|_{H^{\alpha'}}^{\frac{2}{1+\alpha'}} \leq C \|f(t)\|_{H^{-1}}^{\frac{2\alpha'}{1+\alpha'}}   ( \|f^{in}\|_{H^{1}} t)^{\frac{2}{1+\alpha'}},
\end{equation}
which gives \eqref{ts:opt} and concludes the proof.
\end{proof}

\subsection{Preliminaries}\label{ssec:prelim_inv}
We record here some more useful properties of our construction of $u_m$. Let $S(m):=\{y_m^j;\,j=0,\ldots N_m\}$ be the set of grid points at stage $m$, and define
\begin{equation}
 S_0(m):=\{y\in S(m);\,\exists j\in\{0,\ldots,N_{m-1}\}, k\in\{0,\ldots,p-1\}:\,y=y_{m-1}^{j,k}\},
\end{equation}
i.e.\ $S_0(m)$ is the set of grid points at which $u_m$ does not change slope.
We may then further separate the grid points $S(m)\setminus S_0(m)$ into a set $S_1(m)$ that are ``close to'' grid points of the previous stage $m-1$ and a remainder set $S_2(m)$ of new, ``interior'' grid points. The latter $S_2(m)$ can then be further subdivided into $q$ sets {$A_k(m)$ ($0\leq k\leq q-1$)} of grid points, with the property that intervals of length $(q+1)\ell_m$ and starting point in $A_k(m)$ do not have any overlaps. The details are as follows.
\begin{lemma}\label{lem:Ssplit}
 Let
 { \begin{equation}
  D:=\{-(q-1),-(q-3),\ldots,0,\ldots,q-3,q-1\},
 \end{equation}
 and define the disjoint sets $S_1(m), S_2(m)$ as
 \begin{equation}\label{eq:defS12}
 \begin{aligned}
  S_1(m)&:=\bigcup_{d\in D}\tau_{d\ell_m}(S(m-1)),\\
  S_2(m)&:=S(m)\setminus(S_0(m)\cup S_1(m)),
 \end{aligned} 
 \end{equation}
 where $\tau_c$ denotes the translation operator by $c\in \T$.}
 Then clearly $S(m)=S_0(m)\cup S_1(m)\cup S_2(m)$ and moreover
\begin{enumerate}[label={\rm (P\arabic*)}]
  \item\label{it:NoS1} $\# S_1(m)=q\cdot S(m-1)$,
  \item\label{it:valS1} if $y\in S_1(m)$ then there exists $y^\ast\in S(m-1)$ such that $u_m(y)=u_{m-1}(y^\ast)$,
  \item\label{it:S2} $S_2(m)$ can be written as the (disjoint) union
  \begin{equation}\label{eq:decompS2}
   S_2(m)=\bigcup_{k=0}^{q-1}\bigcup_{y\in A_k(m)}\{y\}\cup\{y+(q+1)\ell_{m}\},
  \end{equation}
   where $A_k(m)\subset S_2(m)$, $0\leq k\leq q-1$, 
   are such that if $y,y'\in A_k(m)$, then $[y,y+(q+1)\ell_m]\cap[y',y'+(q+1)\ell_m]=\emptyset$. Moreover, we note that if $y\in A_k(m)$, then by construction
   \begin{equation}\label{eq:propS2um}
    u_m(y)=u_m(y+(q+1)\ell_m),
   \end{equation}
   and in addition we have
   \begin{equation}\label{eq:propS2umder}
    u_m'(y^+)=-u_m'((y+(q+1)\ell_m)^-),\quad u_m'(y^-)=-u_m'((y+(q+1)\ell_m)^+),
   \end{equation}
   where superscripts $+,-$ denote right and left limits, respectively.
 \end{enumerate}
\end{lemma}

In fact, the decomposition into sets $A_k(m)$ could be achieved with only $\frac{q+1}{2}$ sets -- see Figure \ref{fig:Ssets}. However, to simplify the exposition of the proof we give ourselves some room by allowing for up to $q$ such sets, since this distinction is inconsequential for the arguments that follow in Section \ref{ssec:pf_inv} .

\begin{figure}[h!]
  \includegraphics[scale=1]{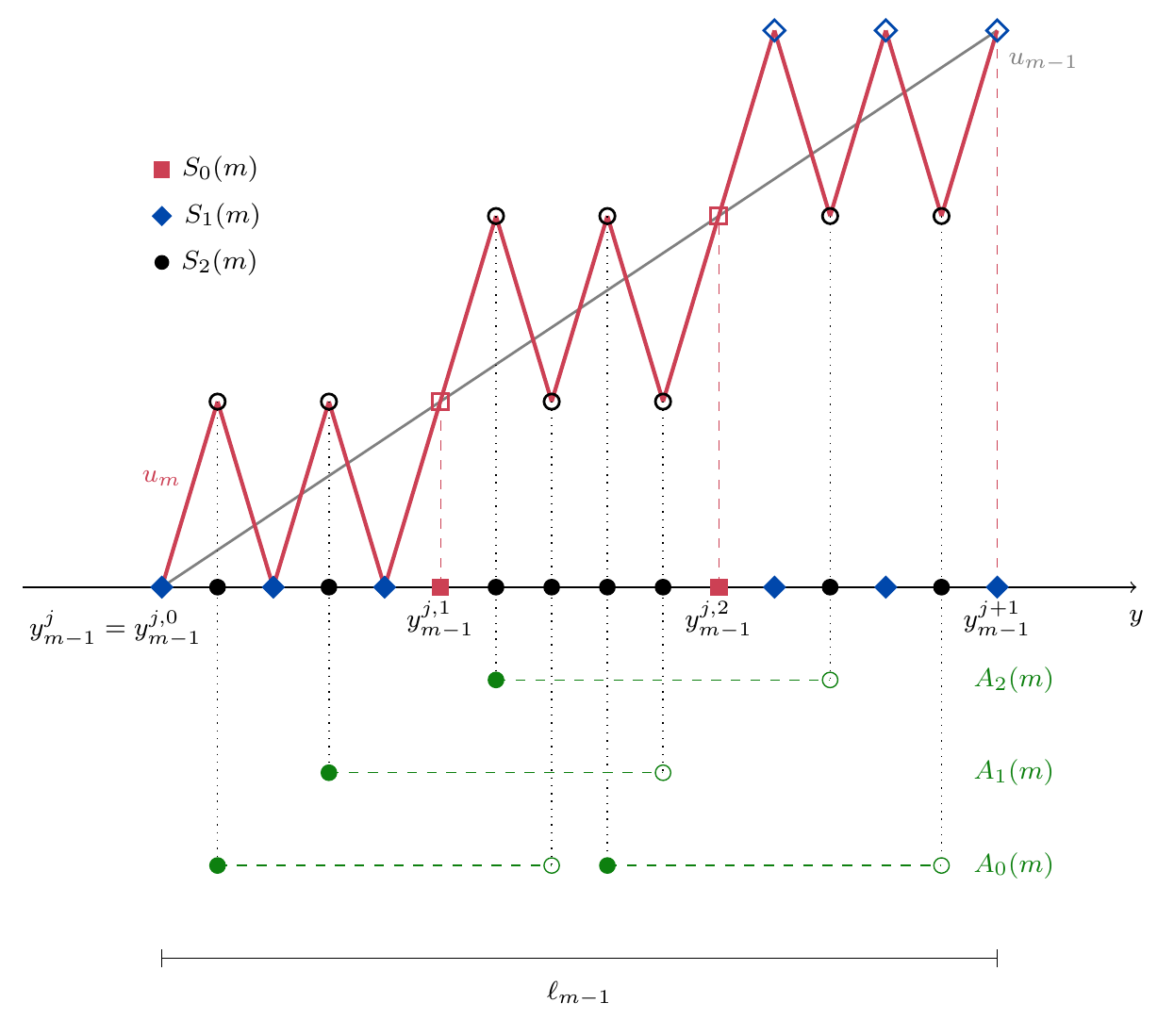}
  \caption{Examples of the sets $S_j(m)$ and $A_k(m)$ of Lemma \ref{lem:Ssplit}, $0\leq j\leq 2$, $0\leq k\leq 3$, in the case where $q=5$, $p=3$.}
  \label{fig:Ssets}
\end{figure}

\begin{proof}
 For $S_1(m)$ defined as in \eqref{eq:defS12}, clearly we have $S_1(m)\subset S(m)$, and by construction the claim \ref{it:NoS1} follows. For \ref{it:valS1} it suffices to observe that if $y\in S_1(m)$, then by construction there exists $d\in D$ such that {$y^\ast:=\tau_{d\ell_m}(y)\in S(m)$}. Since the two grid points satisfy {$\abs{y^\ast-y}\leq\frac{q-1}{2}\ell_m$} it then follows that $u_m(y)=u_{m-1}(y^\ast)$.
 
 We prove the properties \ref{it:S2} of $S_2(m)$ by showing how the sets $A_k(m)$, $0\leq k\leq q-1$, may be constructed. 
 To this end, consider an interval $I_{m-1}^j=[y_{m-1}^j,y_{m-1}^{j+1}]$ and define $A_k(m)$ by requiring that
 \begin{equation}
 \begin{aligned}
  &A_k(m)\cap(I_{m-1}^j\cap S_2(m))=\{y_{m-1}^j+(2k+1)\ell_m+2n h_m;\,n\in\N\}\cap I_{m-1}^j,\quad 0\leq k\leq \frac{q-1}{2},\\
  &A_k(m)\cap(I_{m-1}^j\cap S_2(m))=\{y_{m-1}^j+(2k+2)\ell_m+2nh_m;\,n\in\N\}\cap I_{m-1}^j,\quad \frac{q-1}{2}< k\leq q-1.\\
 \end{aligned} 
 \end{equation}
 One verifies directly that the decomposition \eqref{eq:decompS2} holds.
 
Finally, observe that if $y\in I_{m-1}^j\cap S_2(m)$ there exist $k\in\{0,\ldots,p-1\}$ and $a\in\{0,\ldots,\frac{q-1}{2}\}$ such that $y=y_{m-1}^{j,k}+(2a+1)\ell_m$, and hence by \eqref{eq:defum} there holds
\begin{equation}
 u_m(y)=u_{m-1}(y_{m-1}^{j,k+1}).
\end{equation}
Moreover, we have
 \begin{equation}
  y+(q+1)\ell_m=y+h_m+\ell_m=y_{m-1}^{j,k+1}+2(a+1)\ell_m,
 \end{equation}
so that by the definition \eqref{eq:defum} of $u_m$ it follows that
\begin{equation}
 u_m(y+(q+1)\ell_m)=u_{m-1}(y_{m-1}^{j,k+1}),
\end{equation}
thus proving the claim \eqref{eq:propS2um}. Equation \eqref{eq:propS2umder} follows by noting that if $y$ is a local maximum, then $y+(q+1)\ell_m$ is a local minimum, and vice versa.

This completes the proof.
\end{proof}

\subsection{Proof of Proposition \ref{prop:inv_decay}}\label{ssec:pf_inv}
Towards the proof of Proposition \ref{prop:inv_decay} we note that the equation \eqref{eq:pass_scal} decouples in $x$, so we may write
\begin{equation}\label{eq:Fdecomp}
 f(t,x,y)=\sum_{k\in\Z} f_k(t,y)\e^{itkx}.
\end{equation}
Then $f$ satisfies \eqref{eq:pass_scal} iff there holds
\begin{equation}\label{eq:Fpass_scal}
 \de_t f_k+iku(y) f_k=0,\qquad k\in\Z,
\end{equation}
which has solution $f_k(t,y)=\e^{-itku(y)}f_k^{in}(y)$. Thus there holds that
\begin{equation}\label{eq:H-1est}
\begin{aligned}
 \norm{f(t)}_{L^2_xH^{-1}_y}^2&= \sum_{k\in\Z}\norm{f_k(t)}_{H^{-1}_y}^2=  \sum_{k\in\Z}\sup_{\psi\in H^1_y,\norm{\psi}_{H^1}\leq 1}\abs{\ip{f_k(t),\psi}}_{L_y^2}^2\\
 &=\sum_{k\in\Z}\left(\sup_{\psi\in H^1_y,\norm{\psi}_{H^1}\leq 1}\abs{\int_\T \e^{-itku(y)}f^{in}_k(y)\psi(y)\dd y}\right)^2.
\end{aligned}
\end{equation}

\subsubsection{Proof of Part \eqref{it:1}}
To prove the general decay result \eqref{eq:inv_decay_gen} of Proposition \ref{prop:inv_decay} it thus suffices to invoke the below Lemma \ref{lem:inv_decay} in \eqref{eq:H-1est}, which gives
\begin{equation}
 \abs{\int_\T \e^{itku(y)}f^{in}_k(y)\psi(y)\dd y}\leq C (t\abs{k})^{-1}\norm{f^{in}_k\psi}_{W_y^{1,1}}\leq C (t\abs{k})^{-1}\norm{f^{in}_k}_{H_y^1}\norm{\psi}_{H_y^1}.
\end{equation}
This shows that
\begin{equation}
 \norm{f(t)}_{L^2_xH^{-1}_y}^2\leq Ct^{-2}\sum_{k\in\Z{\setminus\{0\}}}\abs{k}^{-2}\norm{f^{in}_k}_{H_y^1}^2\leq Ct^{-2}\norm{f^{in}}_{H^{-1}_xH^1_y}^2, 
\end{equation}
which is the claim. Thus it remains to establish the following:

\begin{lemma}\label{lem:inv_decay}
 There exists a constant $C:= C(p,q)>0$ such that
\begin{equation}\label{ts:inv_decay}
  \int_\T \e^{itu_m(y)}\varphi(y)\dd y\leq Ct^{-1}\norm{\varphi}_{W_y^{1,1}}.
 \end{equation}
\end{lemma}

\begin{proof}
Since from \eqref{eq:convest} 
we have that $\norm{u-u_m}_{L^\infty}\leq p^{-m}$, it suffices in fact to show that \emph{uniformly in $m$} there holds that
 \begin{equation}
  \int_\T \e^{itu_m(y)}\varphi(y)\dd y\leq Ct^{-1}\norm{\varphi}_{W^{1,1}}.
 \end{equation}
To this end we integrate by parts on the intervals $I_m^{j}$, $0\leq j\leq N_m$, where $u_m$ is linear (and thus $u_m''=0$) and obtain
\begin{equation}
  \int_\T \e^{itu_m(y)}\varphi(y)\dd y=\int_\T \e^{itu_m(y)}\frac{\varphi'(y)}{tu'_m}\dd y+\sum_{j=0}^{N_m-1} \left[\frac{\e^{itu_m(y)}}{tu_m'(y)}\varphi(y)\right]_{y=y_m^{j}}^{y_m^{j+1}}.
\end{equation}
Since by \eqref{eq:umder} we have $\abs{u_m'}=\pi^{-1}q^m$ it follows that
\begin{equation}
 \abs{\int_\T \e^{itu_m(y)}\frac{\varphi'(y)}{tu'_m}\dd y}\leq\frac{\pi}{t}{q^{-m}}\norm{\varphi'}_{L^1}.
\end{equation}
Moreover, for the boundary terms we note that 
\begin{equation}
\begin{aligned}
 T_m(\varphi):=\sum_{j=0}^{N_m-1} \left[\frac{\e^{itu_m(y)}}{tu_m'(y)}\varphi(y)\right]_{y=y_m^{j}}^{y_m^{j+1}}&=\left(\sum_{y\in S_1(m)}+\sum_{y\in S_2(m)}\right)\left(\frac{\e^{itu_m(y)}}{tu_m'(y^+)}\varphi(y)+\frac{\e^{itu_m(y)}}{tu_m'(y^-)}\varphi(y)\right)\\
 &=:T_m^1(\varphi)+T_m^2(\varphi),
\end{aligned} 
\end{equation}
since all grid points in $S(m)$ appear twice in the sum, once as upper and once as lower boundary points of the integrals. Here, for points in $S_1(m)\cup S_2(m)$ both terms have the same sign (either lower boundary with positive slope $s_m$ and upper boundary with negative slope $-s_m$ or the other way around), whereas for points in $S_0(m)$ the slope does not change and hence they do not contribute.

We prove now by induction on $m\in \N$ that for every $\varphi \in W^{1,1}$ we have
\begin{equation}
\begin{aligned}\label{eq:ind-est}
 \abs{T_m(\varphi)} &\leq \frac{4 \pi}{ t}\norm{\varphi}_{W^{1,1}}+\sum_{k=1}^m\frac{q}{tq^k}\norm{\varphi'}_{L^1}.
\end{aligned}
\end{equation}
The conclusion \eqref{ts:inv_decay} follows then by setting $C= 4 \pi+\frac{q}{q-1}$.
Indeed, for $m=0$, this is a consequence of a direct computation and of the fact that $\sup \varphi \leq \norm{\varphi}_{W^{1,1}}$, namely
\begin{equation}
\begin{aligned}
 \abs{T_0(\varphi)} &= \left[\frac{\e^{itu_0(y)}}{tu_0'(y)}\varphi(y)\right]_{y=-\pi}^{0}+  \left[\frac{\e^{itu_0(y)}}{tu_0'(y)}\varphi(y)\right]_{y=0}^{\pi}
\\& = \frac{1}{ts_0} \big(-\varphi(\pi)+ 2 \e^{i} \varphi(0) - \varphi(-\pi)\big) \leq \frac{4 \pi}{t}\norm{\varphi}_{W^{1,1}}.
\end{aligned}
\end{equation}
Now we assume the estimate \eqref{eq:ind-est} for $m-1$ and we prove it for $m$. By property \ref{it:valS1} we have
\begin{equation}
\begin{aligned}\label{eq:T1}
 T_m^1(\varphi)&=\sum_{y\in S_1(m)}\left(\frac{\e^{itu_m(y)}}{tu_m'(y^+)}\varphi(y)+\frac{\e^{itu_m(y)}}{tu_m'(y^-)}\varphi(y)\right)\\
 &=\sum_{d\in D}\sum_{y\in\tau_{d\ell_m}(S(m-1))}\left(\frac{\e^{itu_m(y)}}{tu_m'(y^+)}\varphi(y)+\frac{\e^{itu_m(y)}}{tu_m'(y^-)}\varphi(y)\right)\\
 &=\sum_{d\in D}\sum_{y\in S(m-1)}\left(\frac{\e^{itu_m(\tau_{-d\ell_m}y)}}{tu_m'(\tau_{-d\ell_m}y^+)}\varphi(\tau_{-d\ell_m}y)+\frac{\e^{itu_m(\tau_{-d\ell_m}y)}}{tu_m'(\tau_{-d\ell_m}y^-)}\varphi(\tau_{-d\ell_m}y)\right)\\
 &=\sum_{d\in D}\sum_{y\in S(m-1)}\left(\frac{\e^{itu_{m-1}(y)}}{q\cdot tu_{m-1}'(y^+)}\varphi(\tau_{-d\ell_m}y)+\frac{\e^{itu_{m-1}(y)}}{q\cdot tu_{m-1}'(y^-)}\varphi(\tau_{-d\ell_m}y)\right)\\
 &=\sum_{d\in D}\frac{1}{q}T_{m-1}(\varphi\circ\tau_{-d\ell_m}).
\end{aligned} 
\end{equation}
On the other hand, splitting $T^2_m(\varphi)=T^{2,+}_m(\varphi)+T^{2,-}_m(\varphi)$ for terms with $y^+$ resp.\ $y^-$ we invoke \eqref{eq:propS2um} and \eqref{eq:propS2umder} to deduce that
\begin{equation}\label{eq:T2}
\begin{aligned}
 \abs{T_m^{2,+}(\varphi)}&:=\abs{\sum_{y\in S_2(m)}\frac{\e^{itu_m(y)}}{tu_m'(y^+)}\varphi(y)}\\
 &=\abs{\sum_{k=0}^{q-1}\sum_{y\in A_k(m)}\frac{\e^{itu_m(y)}}{tu_m'(y^+)}\varphi(y)+\frac{\e^{itu_m(y+(q+1)\ell_m)}}{tu_m'((y+(q+1)\ell_m)^+)}\varphi(y+(q+1)\ell_m)}\\
 &=\abs{\sum_{k=0}^{q-1}\sum_{y\in A_k(m)}\frac{\e^{itu_m(y)}}{tu_m'(y^+)}[\varphi(y)-\varphi(y+(q+1)\ell_m)]}\\
 &\leq \sum_{k=0}^{q-1}\frac{1}{tq^m}\sum_{y\in A_k(m)}\abs{\int_y^{y+(q+1)\ell_m}\varphi'(z)\dd z}\\
 &\leq \frac{q}{tq^m}\norm{\varphi'}_{L^1},
\end{aligned}
\end{equation}
where in the last inequality we used the disjointness of the intervals $[y,y+(q+1)\ell_m]$ and $[y',y'+(q+1)\ell_m]$ for $y,y'\in A_k(m)$, by property \ref{it:S2} of Lemma \ref{lem:Ssplit}. Similarly we have $\abs{T^{2,-}_m(\varphi)}\leq \frac{q}{tq^m}\norm{\varphi'}_{L^1}$.

{From \eqref{eq:T1} and \eqref{eq:T2}, since $\#D=q$ and using the inductive assumption  \eqref{eq:ind-est} for $m-1$ to estimate $T_{m-1}(\varphi\circ \tau_{-d_1\ell_m})$ in the right-hand side of \eqref{eq:T2}, we deduce that
\begin{equation}
\begin{aligned}
 T_{m}(\varphi)&=T_m^1(\varphi)+T_m^2(\varphi)=\sum_{d_1\in D}\frac{1}{q}T_{m-1}(\varphi\circ \tau_{-d_1\ell_m})+T_m^2(\varphi)
 \leq \frac{4 \pi}{ t}\norm{\varphi}_{W^{1,1}}+\sum_{k=1}^m\frac{q}{tq^k}\norm{\varphi'}_{L^1}.
\end{aligned} 
\end{equation}
which proves the inductive estimate \eqref{eq:ind-est}. 
}

\end{proof}

\subsubsection{Proof of Parts \eqref{it:2} and \eqref{it:3}}
In view of the decomposition \eqref{eq:Fdecomp} into Fourier modes, the below Lemma \ref{lem:inv_decay_new} can be applied to give our claims: In order to establish \eqref{eq:inv_decay_lobd}, we let $f^{in}_\star:=\cos(x)$, so that by \eqref{eq:H-1est} the corresponding solution $f_\star(t)$ satisfies
\begin{equation}
\begin{split}
 \norm{f_\star(t)}_{L^2_xH^{-1}_y}^2&=\sum_{k=\pm 1}\left[\sup_{\psi\in H^1_y,\norm{\psi}_{H^1_y}\leq 1}\abs{\int_\T \e^{-itku(y)}\left(f^{in}_\star\right)_k(y)\psi(y)\dd y }\right]^2
 \\&\geq \abs{\int_\T \e^{\pm itu(y)}\dd y }^2 \norm{f^{in}_\star}_{L^2_xH^1_y}^2 , 
\end{split}
\end{equation}
where in the last inequality we used that $\psi= 1$ is an admissible choice in the supremum and that $\left(f^{in}_\star\right)_k(y) \equiv 1$.
Now it suffices to apply \eqref{ts:inv_decay3} with $t=t'_m$ to obtain \eqref{eq:inv_decay_lobd}. Similarly, the bounds \eqref{eq:inv_decay_fast_m} are proved by invoking \eqref{eq:H-1est} and the below \eqref{ts:inv_decay2} with $\varphi=f^{in}_k\cdot \psi$.

\begin{remark}
This argument shows that the particular choice $f^{in}_\star:=\cos(x)$ plays a minor role. For instance, the same estimate works for initial data of the form
$$f^{in}_\star(x,y) = f^{(1)}(x) f^{(2)}(y),$$
where 
$$\|f^{(1)}\|_{L^2_x} = 1, \qquad |f^{(1)}|\geq \frac 14, \qquad \|f^{(2)}\|_{H^1_y} = 1, \qquad \inf_{y}f^{(2)}(y) \geq \frac 14.$$
\end{remark}

The proof of Proposition \ref{prop:inv_decay} is thus completed once we have proved the following:
\begin{lemma}\label{lem:inv_decay_new}
 There exists a constant $C:= C(p,q)>0$ such that for $t_m:= {2\pi p^{m}}$, $t'_m:=\pi p^m$, $m\in \N$, there holds
\begin{align}
  \left|\int_\T \e^{it_m u(y)}\varphi(y)\dd y \right|&\leq Ct_m^{-\frac 1\alpha}\norm{\varphi'}_{L^{1}},\label{ts:inv_decay2}\\
  \left| \int_\T \e^{it_m'u(y)}\dd y \right| &\geq C(t'_m)^{-1} 
  . \label{ts:inv_decay3}
 \end{align}
\end{lemma}

\begin{proof}
\textbf{Proof of \eqref{ts:inv_decay2}.} 
Consider $u_m$, $\ell_m$ $s_m$, $\{y^j_m\}_{j=1,..., N_m}$ as in Section~\ref{sec:constr}. To prove~\eqref{ts:inv_decay2}, we have to estimate
\begin{equation}\label{eqn:fast1}
\int_\T \e^{itu(y)}\varphi(y)\dd y
=\int_\T \e^{itu_m(y)}\varphi(y)\dd y + \int_\T \big( \e^{itu(y)}- \e^{itu_m(y)})\varphi(y)\dd y =: (I)+ (II).
\end{equation}

To estimate the first term, we integrate by parts on each of the  segments $[{y^{j-1}_m},{y^{j}_m}]$, ${j=1,...,N_m}$, where $u_m$ is linear of slope $s_m$, which yields
\begin{equation}\label{eqn:fast2}
\begin{split}
(I) &= \sum_{j=1}^{N_m} \frac{1}{t u_m'(y^{j-}_m)}\int_{y^{j-1}_m}^{y^{j}_m} \partial_y(  \e^{itu_m(y)}) \varphi(y)\dd y
\\
&= \sum_{j=1}^{N_m} \frac{1}{t u_m'(y^{j-}_m)} \Big(\int_{y^{j-1}_m}^{y^{j}_m}  \e^{itu_m(y)} \varphi'(y)\dd y + \e^{itu_m(y^{j}_m)} \varphi(y^{j}_m) -  \e^{itu_m(y^{j-1}_m)} \varphi(y^{j-1}_m) \Big).
\end{split}
\end{equation}
We now observe that $t_{{m}} |u_m'(y^{j-}_m)| = 2(pq)^{m}=2 p^{\frac m \alpha}= 2^{1-\frac 1 \alpha} ( \pi^{-1})^{\frac 1 \alpha}  t_{{m}}^{\frac 1 \alpha}$
and that we have 
\begin{equation}
\label{eqn:osc-um}
|t_{{m}} u_m(y^{j}_m) - t_{{m}}u_m(y^{j-1}_m)| = t_{{m}} s_m \ell_m =t_{{m}} \frac{1}{p^m}  =2\pi,
\end{equation}
so that 
$\e^{it_{{m}}u_m(y^{j}_m)} = \e^{it_{{m}}u_m(y^{j-1}_m)}$.
Hence 
\begin{equation}\label{eqn:fast3}
\begin{split}
 (I) &\leq {C}t_{{m}}^{-\frac 1 \alpha}  \sum_{j=1}^{N_m} \Big(\int_{y^{j-1}_m}^{y^{j}_m} |\varphi'(y)| \dd y
+ | \varphi(y^{j}_m) -  \varphi(y^{j-1}_m)| \Big)
\leq  {C}t_{{m}}^{-\frac 1 \alpha}  \| \varphi'\|_{L^1}.
\end{split}
\end{equation}
Next, let $\{ y^{j,k}_m\}_{j={0},..., N_m, k=0,...,p}$ be as in Section~\ref{sec:constr}. We claim that for every $j=0,..., {N_m}-1$ and for every $h_{m+1}=\ell_m p^{-1}$-periodic function $f: [y^{j}_m,y^{j+1}_m] \to \R$, namely
 such that
$$f(x+ y^{j}_{m})= f(x+ y^{j,k}_{m}) \qquad \mbox{for every } x\in [0, \ell_{m} p^{-1}], \;{k=0,...,p-1,} $$
 we have 
\begin{equation}
\label{eqn:meanzerof}
\int_{y^{j}_m}^{y^{j{+1}}_m}  \e^{it_{{m}} (u_m(y)+f(y))} \dd y = 0.
\end{equation}
To prove this claim, {observe that by periodicity of $f$}
\begin{equation}
\begin{split}
\int_{y^{j}_m}^{y^{j+1}_m}  \e^{it (u_m(y)+f(y))} \dd y &= \sum_{k=0}^{p-1}\int_{y^{j,k}_m}^{y^{j,k{+1}}_m}  \e^{it (u_m(y)+f(y))} \dd y \\
&= \sum_{k=0}^{p-1}\int_{0}^{\ell_m p^{-1}}  \e^{it (u_m(y+y^{j,k}_m)+f(y+y^{j,k}_m))} \dd y
\\
&= \int_{0}^{\ell_m p^{-1}}   \e^{itf(y)} \sum_{k=0}^{p-1}\e^{it u_m(y+y^{j,k}_m)} \dd y.
\end{split}
\end{equation}
Since $t_{{m}} u_m:[y^{j}_m,y^{j+1}_m] \to \R$ is a linear function  such that $t_{{m}} u_m(y^{j+1}_m) - t_{{m}} u_m(y^j_m)= \pm2\pi$, we have
$$\sum_{i=0}^{p-1}\e^{it_{{m}} u_m(y+y^{j,k}_m)}= \sum_{k=0}^{p-1}\e^{it_{{m}} u_m(y)+ i{k} \frac{2\pi}p}=0$$
and this completes the proof of \eqref{eqn:meanzerof}.

For any given $j =1,.., N_m$ we observe by an inductive argument that $(u_{m'}-u_m)$ is a $\ell_m p^{-1}$-periodic function in $[y^{j-1}_m,y^{j}_m]$ for every $m'\geq m$.
Indeed, this is trivial for $m'=m$, and it is verified by the explicit construction of $u_{m'+1}$ once $u_{m'}$ satisfies this property. Hence, applying \eqref{eqn:meanzerof} to $f=0$ and to $f= t_{{m}}(u_{m'}-u_m):[y^{j-1}_m,y^{j}_m] \to \R$ we deduce that
\begin{equation}
\label{eqn:meanzero}
\int_{y^{j-1}_m}^{y^{j}_m}  \e^{it_{{m}}u_m(y)} \dd y =  \int_{y^{j-1}_m}^{y^{j}_m}  \e^{it_{{m}}u(y)} \dd y = 0.
\end{equation}
Coming back to (II), we rewrite it using \eqref{eqn:meanzero} and then we estimate by means of Fubini's theorem
\begin{equation}\label{eqn:fast4}
\begin{split}
(II) &=    \sum_{j=1}^{N_m} \int_{y^{j-1}_m}^{y^{j}_m} \big( \e^{itu(y)}- \e^{itu_m(y)})(\varphi(y)-\varphi(y^{j-1}_m))\dd y
\\
&\leq    \sum_{j=1}^{N_m} \int_{y^{j-1}_m}^{y^{j}_m} |\varphi(y)-\varphi(y^{j-1}_m)|\dd y
\\
&\leq  \sum_{j=1}^{N_m} \int_{y^{j-1}_m}^{y^{j}_m} \int 1_{\{y^{j-1}_m \leq z \leq y\}} |\varphi'(z)| \dd z \dd y
\\
&\leq \ell_m \sum_{j=1}^{N_m} \int_{y^{j-1}_m}^{y^{j}_m} |\varphi'(z)| \dd z  =
( {2}\pi)^{-\frac 1 \alpha}  t_{{m}}^{-\frac 1 \alpha}  \| \varphi'\|_{L^1}.
\end{split}
\end{equation}
In the last equality we used that $\ell_m = \pi (pq)^{-m} = \pi p^{-\frac m\alpha}= 2^{-\frac 1 \alpha} ( \pi^{-1})^{\frac 1 \alpha}  t^{-\frac 1 \alpha}$.
From \eqref{eqn:fast1}, \eqref{eqn:fast3} and \eqref{eqn:fast4} we obtain \eqref{ts:inv_decay2}.

\textbf{Proof of \eqref{ts:inv_decay3}.}
We start computing the left-hand side in \eqref{ts:inv_decay3}
\begin{equation}\label{eqn:low1}
\int_\T \e^{it_m'u(y)}\dd y= \sum_{j=1}^{N_m/2}
 \int_{y^{2j-2}_m}^{y^{2j}_m}  \e^{it_m'u(y)} \dd y.
\end{equation}
To rewrite the right-hand side, we first prove the following claim: for every $m\in \N$, $j\in \{0,..., N_{m}-1\}$, $k \in \{0,..., p-1\}$ we have
\begin{equation}\label{eqn:low5}
 \int_{y^{j,k}_m}^{y^{j,k+1}_m}  \e^{it_m'u_{m+1}(y)} \dd y
=
 \int_{y^{j,k}_m}^{y^{j,k+1}_m}  \e^{it_m'u_m(y)} \dd y.
\end{equation}
Indeed, let us define the following set of $q$ elements $L_{m,j,k}= \{l: y^{l}_{m+1}\in[{y^{j,k}_m},{y^{j,k+1}_m})\}$, and let $l_{min}$ be the minimum element of $L_{m,j,k}$, which corresponds to $y^{l_{min}}_{m+1}={y^{j,k}_m}$. By definition of $u_{m+1}$, we have that {for $y\in[0,\ell_{m+1}]$}
$$u_{m+1}(y^{2l}_{m+1}+y) = u_{m+1} (y^{l_{min}}_{m+1}+y),
\qquad u_{m+1}(y^{2l+1}_{m+1}+y) = u_{m+1} (y^{l_{min}}_{m+1}+ \ell_{m+1}-y)$$
and finally we observe that $ u_{m+1} (y^{l_{min}}_{m+1}+y) = u_m(y^{l_{min}}_{m+1}+q y)$, {since $u_{m+1}$ has $q$ times the slope of $u_m$}.
Hence we can rewrite the left-hand side of \eqref{eqn:low5} as
\begin{equation}
\begin{split}
 \int_{y^{j,k}_m}^{y^{j,k+1}_m}  \e^{it_m'u_{m+1}(y)} \dd y
&= \sum_{l \in L_{m,j,k}} \int_{y^{l}_{m+1}}^{y^{l+1}_{m+1}}\e^{it_m'u_{m+1}(y)} \dd y
= \sum_{l \in L_{m,j,k}} \int_{0}^{\ell_{m+1}}\e^{it_m'u_{m+1}(y^{l}_{m+1}+y)} \dd y
\\&= q \int_{0}^{\ell_{m+1}}\e^{it_m'u_{m+1}(y^{l_{min}}_{m+1}+y)} \dd y
= q \int_{0}^{\ell_{m+1}}\e^{it_m'u_{m}(y^{l_{min}}_{m+1}+q y)} \dd y
\\&=  \int_{0}^{q\ell_{m+1}}\e^{it_m'u_{m}(y^{l_{min}}_{m+1}+y)} \dd y
=
 \int_{y^{j,k}_m}^{y^{j,k+1}_m}  \e^{it_m'u_m(y)} \dd y.
\end{split}
\end{equation}

From \eqref{eqn:low5}, we deduce the following equality by induction on $l$
\begin{equation}\label{eqn:low6}
 \int_{y^{2j-2}_m}^{y^{2j}_m}  \e^{it_m'u_l(y)} \dd y
=
 \int_{y^{2j-2}_m}^{y^{2j}_m}  \e^{it_m'u_m(y)} \dd y \qquad \mbox{for every } l \geq m.
\end{equation}
Indeed, to verify the inductive step we assume that this equality holds for $l$, and we have 
\begin{equation}\label{eqn:low7}
 \int_{y^{2j-2}_m}^{y^{2j}_m}  \e^{it_m'u_{l+1}(y)} \dd y
= \sum_{j,k} \int_{y^{j,k}_l}^{y^{j,k+1}_l}  \e^{it_{{m}}'u_{l+1}(y)} \dd y
=\sum_{j,k} \int_{y^{j,k}_l}^{y^{j,k+1}_l} 
 \e^{it_m'u_{l}(y)} \dd y
=
 \int_{y^{2j-2}_m}^{y^{2j}_m}  \e^{it_m'u_m(y)} \dd y,
\end{equation}
where the sum is taken over all $j, k$ such that $y^{j,k}_l \in [{y^{2j-2}_m},{y^{2j}_m})$. This proves \eqref{eqn:low6}.
Letting $l \to \infty$ in \eqref{eqn:low6}, we find that each term in the right-hand side of \eqref{eqn:low1} can be rewritten only in terms of $u_m$, instead of $u$
\begin{equation}\label{eqn:low2}
 \int_{y^{2j-2}_m}^{y^{2j}_m}  \e^{it_m'u(y)} \dd y
=
 \int_{y^{2j-2}_m}^{y^{2j}_m}  \e^{it_m'u_m(y)} \dd y.
\end{equation}
By \eqref{eqn:osc-um} and since $t_m' =t_m/2$ we know that 
\begin{equation}
\label{eqn:oscumprimo}
|t_m'u_m(y^{j+1}_m) - t_m'u_m(y^{j}_m)| =\pi \qquad \forall j=0,..., N_m-1 .
\end{equation}
so that, since $t_m' u_m$ is linear in $[y^{j}_m,y^{j+1}_m]$, 
\begin{equation}
\label{eqn:lin-explicit}
t_m' u_m(y^j_m + \ell_m y) =  t_m'u_m(y^{j}_m) +\frac{u_m'(y^{j+}_m)} {|u_m'(y^{j+}_m)|} \pi y,\qquad {y\in [0,1]}.
\end{equation}
Since $u_m(-\pi) = u_m (y^0_m)=0,$ \eqref{eqn:oscumprimo} implies that  ${t'_m}u_m (y^j_m)= a_j \pi$, with $a_j $ an integer number which is even if and only if $j$ is even. In other words, we have 
\begin{equation}\label{eqn:pm1}
\e^{it_m'u_m(y^{j}_m) }= \begin{cases}
1 \qquad & \mbox{if } j \mbox{ is even}
\\-1\qquad & \mbox{if } j \mbox{ is odd.}
\end{cases}
\end{equation}
Hence, by a change of variable and \eqref{eqn:lin-explicit}, we have {with \eqref{eqn:pm1} that}
\begin{equation}
 \begin{split}
 \int_{y^{j-1}_m}^{y^{j}_m}  \e^{it_m'u_m(y)} \dd y
&=
\ell_m \int_{0}^{1}  \e^{it_m' u_m(y^j_m + \ell_m y)} \dd y
\\&=
\ell_m\e^{i t_m'u_m(y^{j}_m)} \int_{0}^{1}  \e^{{i}\frac{u_m'(y^{j+}_m)} {|u_m'(y^{j+}_m)|} \pi y} \dd y
=\frac{2i \ell_m}{\pi} \frac{u_m'(y^{j+}_m)} {|u_m'(y^{j+}_m)|} \e^{i t_m'u_m(y^{j}_m)} .
 \end{split}
\end{equation}
{(}Notice that, in view of \eqref{eqn:pm1}, the last two terms in this product are $\pm 1$.{)} By applying this formula on two consecutive intervals, we obtain
\begin{equation}\label{eqn:low3}
 \int_{y^{2j-2}_m}^{y^{2j}_m}  \e^{it_m'u_m(y)} \dd y
=
\begin{cases}
0 \qquad&\mbox{ if }u_m'(y^{2j-2}_m) = u_m'(y^{2j-1}_m)
\\ 4i \frac 1 {(pq)^m} \qquad&\mbox{ if }u_m'(y^{2j-2}_m) =s_m= -u_m'(y^{2j-1}_m)
\\-4i \frac 1 {(pq)^m}  \qquad&\mbox{ if }u_m'(y^{2j-2}_m) =-s_m = -u_m'(y^{2j-1}_m).
\end{cases}
\end{equation}
Therefore, from \eqref{eqn:low1}, \eqref{eqn:low2} and \eqref{eqn:low3} we get
\begin{equation}\label{eqn:boh}
\begin{split}
\int_\T \e^{itu(y)}\dd y&=
4i \frac 1 {(pq)^m}  M_m,
\end{split}
\end{equation}
where 
$$M_m=  \#\{j: u_m'(y^{2j-2}_m) =s_m = -u_m'(y_m^{2j-1} \} - \#\{j: u_m'(y^{2j-2}_m) =-s_m = -u_m'(y_m^{2j-1} \} .$$
Hence we are left to estimate $M_m$ inductively on $m$; we show that
\begin{equation}
\label{eqn:Mm}
M_m = {C}q^m.
\end{equation}
Indeed, suppose that a certain double interval {$[y_{m-1}^{2j-2},y_{m-1}^{2j}]$} in $u_{m-1}$ contains a positive slope followed by a negative slope. When we split these two intervals into $2p$ subintervals for the construction of $u_{m}$, the first $p-1$ as well as the last $p-1$ do not contribute in $M_{m}$, namely each pair of consecutive intervals contains exactly the same number of elements in the first and in the second set in the definition of $M_m$. In turn, the central two intervals contribute with $q$ elements of the first set. 

By symmetry, we obtain the same result when a certain doubled interval in $u_{m-1}$ contains a negative slope followed by a positive slope. Finally, the {analogous} argument can be performed in the case where in certain doubled interval in $u_{m-1}$ there are two slopes of the same sign. 

From \eqref{eqn:boh} and \eqref{eqn:Mm} we conclude \eqref{ts:inv_decay3}.
\end{proof}

\section{Enhanced dissipation}\label{sec:enhanced}
Our control on the enhanced dissipation is most precisely stated as estimates for the semigroup of the linear operator
\begin{equation}\label{eq:fadsds}
 \cL_\nu:=-u\de_x+\nu\de_{yy}
\end{equation}
closely related to equation \eqref{eq:passcal}. {Indeed, notice that
\begin{align}
\| \e^{(\nu\de_{xx}+\cL_\nu) t}P_k  \|_{L^2\to L^2}\leq \| \e^{\cL_\nu t}P_k  \|_{L^2\to L^2} \qquad \forall t\geq0,
\end{align}
so only the semigroup generated by \eqref{eq:fadsds} needs to be considered.}
\begin{proposition}\label{prop:enhdissip}
There exist positive constants $\eps, C>0$ such that for every $\nu>0$ and every integer $k\neq 0$
satisfying $\nu|k|^{-1}\leq 1/2$ we have
\begin{align}
\| \e^{\cL_\nu t}P_k  \|_{L^2\to L^2}\leq C\e^{-\eps\lambda_{\nu,k}t}, \qquad \forall t\geq0,
\end{align}
where $P_k$ denotes the projection to the $k$-th Fourier mode in $x$ and
\begin{align}
\lambda_{\nu,k}=\nu^\frac{\alpha}{\alpha+2}|k|^\frac{2}{\alpha+2}
\end{align}
is the decay rate.
\end{proposition}

{The result is a consequence of the following general Gearhart-Pr\"uss theorem, established in \cite{WEI18}.}

\begin{theorem}[\cite{WEI18}*{Theorem 1.3}]
{Let $H$ be an m-accretive operator in a Hilbert space $X$, with domain $D(H)\subset X$. Then $\| \e^{-tH}\|_{X\to X}\leq \e^{-t\Psi(H)+\pi/2}$, where
\begin{align}
\Psi(H)=\inf\left\{\|(H-i\lambda)g\|_X: g\in D(H),\,\lambda\in \R,\, \|g\|_X=1\right\}.
\end{align}}
\end{theorem}
{The quantity $\Psi(H)$ above is related to the \textit{pseudo-spectral} properties of the operator $H$, see \cite{GGN09}. In our particular case, for which $H=-\cL_\nu$ and $X=P_k L^2$, a direct way to 
estimate the pseudo-spectrum is provided in  \cite{WEI18}*{Theorem 5.1}.} The setup is as follows: Let $\psi$ be the stream function associated to $u$, such that
\begin{align}
\psi'(y)=-u(y), \qquad \int_{\T}\psi(y)\dd y=0.
\end{align}
As in \cite{WEI18}, for $\delta\in(0,1)$ we define 
\begin{align}
\omega_1(\delta,u)=\inf_{\bar{y},c_1,c_1\in \R}\int_{\bar{y}-\delta}^{\bar{y}+\delta} \abs{\psi(y)-c_1-c_2y}^2\dd y.
\end{align}
Then according to \cite{WEI18}*{Theorem 5.1}, our Proposition \ref{prop:enhdissip} follows if we can show that there exist $C_1>0$ such that for all $\delta\in(0,1)$ there holds
\begin{equation}\label{eq:omega1}
 \omega_1(\delta,u)\geq C_1\delta^{2\alpha+3}.
\end{equation}
{ As shown in \cite{WEI18}, for $\delta= (\nu/\lambda_{\nu,k})^{1/2}<1$, the  quantity $\omega_1(\delta,u)$ then provides the desired lower bound $\Psi(\cL_\nu) \gtrsim \lambda_{\nu,k}$.} 
\begin{proof}[Proof of \eqref{eq:omega1}]
For any function
$\varphi\in C(\T)$ and any $h\in \R$, we define a difference operator $\Delta_h^2:C(\T)\to C(\T)$ by
\begin{align}
\Delta_h^2\varphi(y):=\varphi(y)-2\varphi(y+h)+\varphi(y+2h),
\end{align}
analogous to that in \cite{WEI18}.
Note that $\Delta_h^2$ is a linear operator, and by construction it vanishes on affine functions, i.e.
\begin{align}\label{eq:Delta2_aff}
\Delta_h^2[ay+b]=0,\qquad a,b\in\R.
\end{align}
Moreover, if $\varphi$ is $h$-periodic, then 
\begin{align}
\Delta_h^2\varphi=0.
\end{align}

This difference operator allows us to obtain lower bounds for $\omega_1$ as follows: For $\delta\in(0,1)$ and $\bar{y}\in\T$ given, choose $m\in\N$ such that $2\ell_{m}\leq\delta<2\ell_{m-1}$. Then there exists $y_m^j$ such that the interval $J=[y^m_j,y^m_{j+1}]\subset [\bar{y}-\delta,\bar{y}+\delta]$. We note that $u_m$ is linear on $J$ and without loss of generality has slope $s_m$ (else $-s_m$). Recall that $J$ has length $\abs{J}$, let 
\begin{align}
h:=h_{m+1}=\frac{\ell_m}{p},
\end{align}
and observe that $J':=[y_m^j,y_m^j+3h]\subset J$ since $p\geq 3$.
In addition (and crucially), for any $k\geq m$, the function $\psi_{k+1}-\psi_k$ is
$h$-periodic on $J$. Therefore, 
\begin{align}\label{eq:psi-psim}
\Delta_h^2\psi=\Delta_h^2\psi_m+ \Delta_h^2\left(\sum_{k\geq m} \psi_{k+1}-\psi_k\right)=\Delta_h^2\psi_m, \qquad \text{on } J'.
\end{align}
Moreover, since $\Delta_h^2$ is linear and vanishes on affine functions \eqref{eq:Delta2_aff}, by direct computation we have for $y\in J'$ that
\begin{align}
\Delta_h^2\psi_m (y)=\Delta_h^2[\frac{s_m}{2}y^2]= \frac{s_m}{2}y^2-s_m (y+h)^2 +\frac{s_m}{2}(y+2h)^2=s_m h^2.
\end{align}
It thus follows that
\begin{align}\label{eq:psim_int}
 \int_{y_m^j}^{y_m^j+h}\abs{\Delta_h^2\psi_m (y)}^2\dd y= s_m^2 h^5=C_{p,\alpha}h^{3+2\alpha},
\end{align}
where we used $(pq)^\alpha=p$ to conclude that
\begin{align}
 s_m^2h^{2-2\alpha}=\left(\frac{1}{\pi}\right)^2q^{2m}\cdot \left(\frac{\ell_m}{p}\right)^{2-2\alpha}=\frac{1}{\pi^{2\alpha}p^{2-2\alpha}}q^{2m}(pq)^{m(2\alpha-2)}=\frac{1}{\pi^{2\alpha}p^{2-2\alpha}}=:C_{p,\alpha}.
\end{align}
In particular, from \eqref{eq:psim_int} and \eqref{eq:psi-psim} we deduce that for any $c_1,c_2\in\R$
\begin{equation}\label{eq:omegapprox}
\begin{aligned}
 C_{p,\alpha}h^{3+2\alpha}&=\int_{y_m^j}^{y_m^j+h}\abs{\Delta_h^2\psi_m (y)}^2\dd y=\int_{y_m^j}^{y_m^j+h}\abs{\Delta_h^2\psi (y)}^2\dd y\\
 &=\int_{y_m^j}^{y_m^j+h}\abs{\Delta_h^2\left(\psi (y)-c_1-c_2y\right)}^2\dd y\\
 &\leq \int_{y_m^j}^{y_m^j+3h}\abs{\psi (y)-c_1-c_2y}^2\dd y\\
 &\leq \int_J\abs{\psi (y)-c_1-c_2y}^2\dd y\\
 &\leq \int_{\bar{y}-\delta}^{\bar{y}+\delta}\abs{\psi (y)-c_1-c_2y}^2\dd y.
\end{aligned} 
\end{equation}

Since the above arguments hold for any $\bar{y}\in\T$ and $c_1,c_2\in\R$, we may take the infimum as in the definition \eqref{eq:omega1} of $\omega_1$ and use the fact that $\delta<2\ell_{m-1}=2p^2q\cdot h$ to conclude from \eqref{eq:omegapprox} that
\begin{equation}
 \omega_1(\delta,u)\geq C_{p,\alpha}h^{3+2\alpha}\geq \frac{C_{p,\alpha}}{(2p^2q)^{3+2\alpha}}\cdot\delta^{3+2\alpha}.
\end{equation}
This concludes the proof. 
\end{proof}

\textbf{ Acknowledgements}. 
M.C. thanks Elia Bru\'e for interesting discussions on this topic; she has been supported by
the SNF Grant 182565 ``Regularity issues for the Navier-Stokes equations and for other variational problems".
M.C.Z. acknowledges funding from the Royal Society through a University Research Fellowship (URF\textbackslash R1\textbackslash 191492).

%%      ---------------------------------------------------------------------
%%      --------------------------- BIBLIOGRAPHY ----------------------------
%%      ---------------------------------------------------------------------
%% PUT HERE THE BIBLIOGRAPHY IN YOUR FAVOURITE FORMAT
%% Please check that the format of the bibliography is uniform and coherent

\bibliographystyle{abbrv}
\bibliography{Rough_biblio}

\end{document}